\documentclass[11pt]{amsart}
\usepackage{amssymb,amsmath,amsfonts,amscd,euscript}
\usepackage{hyperref}
\newcommand{\nc}{\newcommand}

\numberwithin{equation}{section}
\newtheorem{thm}{Theorem}[section]
\newtheorem{prop}[thm]{Proposition}
\newtheorem{lem}[thm]{Lemma}
\newtheorem{cor}[thm]{Corollary}
\newtheorem{rem}[thm]{Remark}

\newtheorem{example}[thm]{Example}

\newtheorem*{thma}{Theorem A}

\nc{\gl}{\mathfrak{gl}}
\nc{\GL}{\mathfrak{GL}}
\nc{\g}{\mathfrak{g}}
\nc{\gh}{\widehat\g}
\nc{\h}{\mathfrak{h}}
\nc{\la}{\lambda}
\nc{\al}{\alpha }
\nc{\be}{\beta }
\nc{\ve}{\varepsilon }
\nc{\om}{\omega }

\nc{\ta}{\theta}
\nc{\ch}{{\mathop {\rm ch}}}
\nc{\Tr}{{\mathop {\rm Tr}\,}}
\nc{\Id}{{\mathop {\rm Id}}}
\nc{\ad}{{\mathop {\rm ad}}}
\nc{\bra}{\langle}
\nc{\ket}{\rangle}
\nc{\x}{{\bf x}}
\nc{\bV}{{\bf V}}
\nc{\bs}{{\bf s}}
\nc{\br}{{\bf r}}
\nc{\bb}{{\bf b}}
\nc{\bk}{{\bf k}}
\nc{\bp}{{\bf p}}
\nc{\pa}{\partial}
\nc{\ld}{\ldots}
\nc{\cd}{\cdots}
\nc{\hk}{\hookrightarrow}
\nc{\T}{\otimes}
\nc{\gr}{\mathrm{gr}}
\nc{\ov}{\overline}

\nc{\cO}{\mathcal O}
\nc{\msl}{\mathfrak{sl}}
\nc{\mgl}{\mathfrak{gl}}
\nc{\U}{\mathrm U}
\nc{\V}{\EuScript V}
\nc{\cL}{\mathcal{L}}
\nc{\Res}{\mathrm{Res\ }}

\newcommand{\bC}{{\mathbb C}}
\newcommand{\bQ}{{\mathbb Q}}
\newcommand{\bZ}{{\mathbb Z}}
\newcommand{\bR}{{\mathbb R}}

\newcommand{\fh}{{\mathfrak h}}

\newcommand{\fg}{{\mathfrak g}}

\nc{\Q}{\mathfrak Q}

\begin{document}

\title[Large tensor products and Littlewood-Richardson coefficients]
{Large tensor products and Littlewood-Richardson coefficients}

\author{Evgeny Feigin}
\address{Evgeny Feigin:\newline
Department of Mathematics, HSE University, Russian Federation,
Usacheva str. 6, 119048, Moscow,\newline
{\it and }\newline
Skolkovo Institute of Science and Technology, Skolkovo Innovation Center, Building 3,
Moscow 143026, Russia
}
\email{evgfeig@gmail.com}
\keywords{Lie algebras, tensor products, asymptotic representation theory}

\begin{abstract}
The Littlewood-Richardson coefficients describe the decomposition of tensor products of irreducible representations
of a simple Lie algebra into irreducibles. Assuming the number of factors is large, one gets a measure on the space of weights.
This limiting measure was extensively studied by many authors. In particular, Kerov computed the corresponding density in a special case
in type A and Kuperberg gave a formula for the general case. The goal of this paper is to give a short, self-contained and 
pure Lie theoretic proof of the formula for the density of the limiting measure. Our approach is based on the link between the limiting 
measure induced by the Littlewood-Richardson coefficients and the measure defined by the weight multiplicities of the tensor
products.
\end{abstract}

\maketitle

\section*{Introduction}
In this paper we consider a problem from the asymptotic representation theory, which attracts 
a lot of attention during the last decades (see e.g. \cite{BOO,VK1,VK2,LS,GP,K2,Kup1,Kup2}).
Namely, let $\fg$ be a simple complex finite-dimensional Lie algebra, $\fh\subset\fg$ a Cartan subalgebra, $V_1,\dots,V_k$ -- irreducible 
finite-dimensional $\fg$-modules. 
We are interested in the multiplicities of irreducible 
$\fg$-modules  in the tensor product $\bigotimes_{i=1}^k V_i^{\otimes N_i}$ for large $N_1,\dots,N_k$.
The problem in various regimes was considered in several papers (see e.g. \cite{B1,B2,B3,OC,CS}). 
The square root regime (see below for the precise formulation) we are interested in  
is addressed in \cite{CS,K1,Kup1,NP,TZ,PR}. In particular, Kerov obtained an explicit formula in a very special case \cite{K1}
and Kuperberg \cite{Kup1} gave a general formula (see also \cite{B3,TZ,PR}). Our goal is to give a short proof 
of the limiting density formula in pure representation theoretic terms. Our approach can be used to attack the case of 
representations of infinite-dimensional Lie algebras, such as current algebras and affine Kac-Moody Lie algebras. 

For simplicity in the introduction we consider a special case of the problem: we assume that $\fg$ is 
simply-laced (of type ADE) and $k=1$ (the notation in the general case are more heavy, but conceptually
the special case is of the same level of difficulty). So let $V_\la$ be an irreducible highest weight representation of $\fg$
and let
$
V_\la^{\otimes N}=\bigoplus_{\mu} [V_\la^{\otimes N}:V_\mu]V_\mu
$
be the decomposition into irreducible summands $V_\mu$, where $\mu$ belongs to the set $P_+$ of integral
dominant weights. We consider the sequence of random variables $\eta(N)$ with values of the form $\frac{\mu}{\sigma\sqrt{N}}$
distributed with the following law ($\sigma$ is certain constant explicitly computed in representation-theoretic
terms):
\begin{equation}\label{P1}
P\left(\eta(N)=\frac{\mu}{\sigma\sqrt{N}}\right)=\frac{[V_\la^{\otimes N}:V_\mu]\dim V_\mu}{(\dim V_\la)^N}
\end{equation}
(for the similar constructions see \cite{BG,H,LLP1,LLP2,OO}).
Our goal is to find the $N\to\infty$ limit of the random variables $\eta(N)$. To formulate the result we prepare some notation.
Let $\omega_i$, $i=1,\dots,r$ be the set of fundamental weights, let $\fh_{\ge 0}$ be the $\bR_{\ge 0}$ span of the set 
$\{\om_i\}_{i=1}^r$. Also let $R_+$ be the set of positive roots of $\fg$ summing up to $2\rho$. Finally, let $C$
be the Cartan matrix of $\fg$ and $(\cdot,\cdot)$ be the standard nondegenerate invariant symmetric bilinear form. 
We prove the following theorem (see \cite{B3,CS,Kup1,PR,TZ}).
\begin{thma}
The random variables $\eta(N)$ converge in distribution to the random variable $\eta$ taking values in
$\fh_{\ge 0}$ with the distribution defined by the density
\[
p_\eta(x)=\frac{\sqrt{\det C^{-1}}}{(2\pi)^{r/2}}\frac{\prod_{\al\in R_+} (x,\al)^2}{\prod_{\al\in R_+}(\rho,\al)} \exp(-(x,x)/2), 
\]
where $x=\sum_{i=1}^r x_i\om_i$, $x_i\ge 0$ and the integration (defining the limiting distribution) is performed against the standard
form $dx_1\dots dx_r$. 
\end{thma}

We note that the limiting density does not depend on the representation $V_\la$. 
Recall that in the introduction we assume $\fg$ is simply-laced (in particular, $C$ is symmetric), so 
$(x,x)=\sum_{i,j=1}^r (C^{-1})_{i,j}x_ix_j$.  Also, if $\al=\sum_{i=1}^r l_i\al_i$ is the decomposition of a positive
roots into the sum of simples, then $(x,\al)=\sum_{i=1}^r x_il_i$.

In order to prove Theorem A we consider a simpler sequence of distributions. Namely, we consider another sequence
of random variables $\xi(N)$ with values in $P/(\sigma\sqrt{N})$ ($P$ is the weight lattice) with the following distribution:
\[
P\left(\xi(N)=\frac{\nu}{\sigma\sqrt{N}}\right)= \frac{\dim V_\la^{\otimes N}(\nu)}{(\dim V_\la)^N},
\]
where for a $\fg$-module $W$ and a weight $\nu$ we denote by $W(\nu)$ the $\fh$-eigenspace in $W$ of weight $\nu$.
We prove that
the random variables $\xi(N)$ converge in distribution to the random variable $\xi$ taking values in
$\fh_\bR$ with the distribution defined by the density
\begin{equation}\label{B}
p_\xi(x)=\frac{\sqrt{\det C^{-1}}}{(2\pi)^{r/2}}\exp(-(x,x)/2), 
\end{equation}
where $x=\sum_{i=1}^r x_i\om_i$, $x_i\in\bR$ and the integration is performed against the standard
form $dx_1\dots dx_r$. 

Recall that Kerov \cite{K1} considered the case $\fg=\mgl_n$ and $\la=\om_1$. He showed that the limiting density is given by
the density of the GUE eigenvalues. One easily shows that the restriction of the GUE eigenvalues density to the traceless
matrices produces exactly the density from Theorem A.

Finally, let us briefly comment on the related papers. For the best of our knowledge, Theorem A in the full generality 
first appeared in \cite{Kup1} (the paper mainly treats the type A case in a combinatorial manner). In \cite{B3,TZ}
the authors give asymptotic formulas for the Littlewood-Richardson coefficients (as well as for the weights multiplicities
in the tensor powers of irreducible representations and for certain lattice paths). The paper \cite{CS} describes
the link between the random matrix theory and the tensor products decomposition problem. In the recent paper \cite{PR}
Theorem A is derived from the asymptotic formula for the Littlewood-Richardson coefficients expressed via certain function 
called the large deviation 
rate function. We note that our proof utilizes only the Lie theoretic tools. This allows to obtain the limiting
density without using asymptotic formulas for the Littlewood-Richardson coefficients.

Our paper is organized as follows. In Section \ref{Notation} we collect main notation and definitions from the Lie theory
and probability theory to be used in the main body of the paper. In Section  \ref{WM} we prove \eqref{B}, which is used in
Section \ref{DiI} to prove Theorem A. In Section \ref{TypeA} we work out the type $A$ case explicitly; in particular we show 
the connection with the GUE case. 

We close with several examples for low rank algebras.

Type $A_1$. 
The densities of the random variables $\xi$ and $\eta$ are given by
\begin{gather*}
p_{\xi}(x)=\frac{\sqrt{1/2}}{\sqrt{2\pi}}\exp(-x^2/4), x\in\bR,\\
p_{\eta}(x)=\frac{\sqrt{1/2}}{\sqrt{2\pi}} x^2\exp(-x^2/4), x\in\bR_{\ge 0}.
\end{gather*}

Type $A_2$.
The densities of the random variables $\xi$ and $\eta$ are given by
\begin{gather*}
p_{\xi}(x)=\frac{\sqrt{1/3}}{2\pi}\exp\left(-\frac{x^2+xy+y^2}{3}\right), x,y\in\bR,\\
p_{\eta}(x)=\frac{\sqrt{1/3}}{2\pi}\frac{x^2y^2(x+y)^2}{2}\exp\left(-\frac{x^2+xy+y^2}{3}\right), x,y\in\bR_{\ge 0}.
\end{gather*}

Type $B_2$.
The densities of the random variables $\xi$ and $\eta$ are given by
\begin{gather*}
p_{\xi}(x)=\frac{\sqrt{1/4}}{2\pi}\exp\left(-\frac{x^2+xy+y^2/2}{2}\right), x,y\in\bR,\\
p_{\eta}(x)=\frac{\sqrt{1/4}}{2\pi}\frac{x^2(y/2)^2(x+y)^2(x+y/2)^2}{3/2}\exp\left(-\frac{x^2+xy+y^2/2}{2}\right), x,y\in\bR_{\ge 0}.
\end{gather*}

\vskip 5pt
\noindent
\textbf{Acknowledgments}. We are grateful to Vadim Gorin, Alexander Kolesnikov, Grigori Olshanski, Olga Postnova, and
Nikolai Reshetikhin for useful discussions and correspondence. 
The research was supported by the RSF grant 19-11-00056.

\section{Notation}\label{Notation}
\subsection{Lie algebras} All the material below can be found in the standard textbooks, see e.g. \cite{C,Kir}.

Let $\fg$ be a simple Lie algebra and let $\fh\subset \fg$ be a Cartan subalgebra. 
Let $R=R_+\sqcup R_-$ be the set of roots written as a disjoint union of positive and negative roots and
let $\al_1,\dots,\al_r\in R_+$ be the set of simple roots, where $r=\dim\fh$ is the rank of $\fg$. In particular, 
each root from $R_+$ can be expressed as a linear combination of simple roots with nonnegative integer coefficients.
We denote by $\rho$ the half sum of all the positive roots, $\rho=\frac{1}{2}\sum_{\al>0} \al$, where here and below we write
$\al>0$ for $\al\in R_+$. 

A symmetric bilinear form $(\cdot,\cdot)$ on $\fg$ is called invariant if $([x,y],z)=(x,[y,z])$. It is known that any
two nondegenerate invariant symmetric bilinear forms of $\fg$ are proportional. The restriction of such a form on the Cartan
subalgebra is nondegenerate and hence induces bilinear form on $\fh^*$.  
The Cartan matrix $C=(c_{i,j})_{i,j=1}^r$ is defined by 
$c_{i,j}=2(\al_i,\al_j)/(\al_i,\al_i)$, where $(\cdot,\cdot)$ is an invariant from. In the simply-laced ADE case 
(the lengths of all simple roots squares are equal) 
the Cartan matrix is symmetric. In general, let $d_i=(\al_i,\al_i)/2$. Then the matrix
$\bar C={\rm diag}(d_1,\dots,d_r)C$ is the symmetrized Cartan matrix. There are two natural forms: one is the Killing form
($(x,y)_K={\rm tr}_{\fg} {\rm ad}x{\rm ad}y$) and the standard form defined by the (symmetrized) Cartan matrix, i.e.
$(\al_i,\al_j)=\bar C_{i,j}$. The two forms differ by the known factor $b_\fg$, $b_\fg(\cdot,\cdot)=(\cdot,\cdot)_K$. 
For example, $b_\fg=2(r+1)$ in type $A_r$ and $b_\fg=4r-2$ in type $B_r$ (see \cite{C}, Appendix).

The Weyl group $W$ of $\fg$ is a subgroup of the automorphisms of $\fh^*_{\bR}$ (the real span of the simple roots) 
generated by the simple reflections 
$s_i=s_{\al_i}$. Explicitly, $s_i\beta=\beta-2\al_i(\al_i,\beta)/(\al_i,\al_i)$. For $w\in W$ we denote by $l(w)$ the length of the
element and by $(-1)^w$ the sign $(-1)^{l(w)}$. In particular, $l(s_i)=1$ and $(-1)^{s_i}=-1$ for all $i$. 
For an element $w\in W$ and $\beta\in\fh^*_{\bR}$
the shifted action $w*\beta$ is defined by $w*\beta=w(\beta+\rho)-\rho$. 
The fundamental Weyl alcove in $\fh^*_{\bR}$
consists of $\beta$ such that $(\beta,\al_i)>0$ for all $i=1,\dots,r$. The Weyl group acts freely and transitively 
on the set of all alcoves -- the connected components of $\fh^*_{\bR}$ with all the reflection hyperplanes $\al^\perp$, $\al\in R_+$  
removed. In particular, the number of alcoves coincides with the cardinality of $W$.  
 
The fundamental weights $\om_1,\dots,\om_r\in\fh^*$ form a basis of $\fh^*$ defined by the formula
$(\om_i,\al_j)=d_i\delta_{i,j}$. 
Let $P\subset\fh^*$ be the weight lattice, $P=\bigoplus_{i=1}^r \bZ\om_i$. We define $P_{++}\subset P_+\subset P$
as follows:
\[
P_+=\bigoplus_{i=1}^r \bZ_{\ge 0}\om_i,\quad P_{++}=\bigoplus_{i=1}^r \bZ_{>0}\om_i.
\]  

Irreducible finite-dimensional highest weight $\fg$-modules are classified by the dominant integral weights, i.e. by 
$\la\in P_+$. Given such a $\la$, we denote by $V_\la$ the corresponding irreducible representation of $\fg$. 
$V_\la$ can be decomposed into the direct sum of $\fh$ eigenspaces, $V_\la=\sum_{\mu\in \fh^*} V_\la(\mu)$.
In particular, $V_\la(\la)$ is one-dimensional (spanned by the highest weight vector). For a $\fg$-module $V$ 
with the $\fh$-eigenspaces decomposition $V=\sum_{\mu\in \fh^*} V(\mu)$ the character 
$\ch V$ is defined as the formal sum
\[
\ch V=\sum_{\mu\in \fh^*} z^\mu \dim V(\mu). 
\]  
In particular, $\ch V|_{z=1}=\dim V_\la$. 
An important feature of the characters is that $\ch (V\T W)=\ch V \ch W$.
\begin{rem}
One usually uses the notation $e^\mu$ instead of $z^\mu$. However, we will use the symbol $e$ to denote the 
Euler number, so we use $z$ in the characters instead to avoid confusions. 
\end{rem}

The Weyl character formula says that
\[
\ch V_\la=\frac{\sum_{w\in W} (-1)^wz^{w*\la}}{\prod_{\al>0} (z^{\al/2}-z^{-\al/2})} = 
\frac{\sum_{w\in W} (-1)^wz^{w(\la+\rho)}}{\prod_{\al>0} (1-z^{\al})}.
\] 
The formula implies two formulas: the product form for the denominator
\[
\sum_{w\in W} (-1)^w z^{w\rho}=\prod_{\al>0} (z^{\al/2}-z^{-\al/2})
\]
and the Weyl dimension formula 
\[
\dim V_\la=\frac{\prod_{\al>0} (\la+\rho,\al)}{\prod_{\al>0} (\rho,\al)}.
\] 

Finally, we recall the Casimir element. Let $x_1,\dots,x_{\dim\fg}$ and $x^1,\dots,x^{\dim\fg}$ be dual bases of $\fg$ with respect 
to the Killing form, i.e. $(x_i,x^j)=\delta_{i,j}$. The Casimir element $C_2$ belongs to the universal enveloping
algebra of $\fg$ and is given by the formula  $\sum_{i=1}^{\dim\fg} x_ix^i$ (the result does not depend on the choice of the bases).
The Casimir element belongs to the center of the universal enveloping algebra and hence acts as a scalar operator
in any irreducible $\fg$-module. The constant is equal to $(\la,\la+2\rho)$ for irreducible representation $V_\la$.

\subsection{Probability} 
In the main body of the paper we consider several measures on the real weight space $\fh^*_{\bR}$ (the $\bR$ span of 
the simple roots).
Let $S$ be a finite set in $\fh^*_\bR$ and let $(a_s)_{s\in S}$ be a set of nonnegative real numbers summing up
to one. We denote by $\mu_S$ the discrete measure being the sum of delta measures $\sum_{s\in S} a_s\delta_s$. 
In particular, if $f$ is a function on $\fh^*_\bR$, then $\int_{\fh^*_\bR} fd\mu_S=\sum_{s\in S} a_s f(s)$. 

Let $\mu$ be an absolutely continuous measure with density $p(x)$. Let $S(N)$, $N\ge 1$ be a sequence of finite
sets of $\fh^*_\bR$ and assume that we have attached the weights $(a_s)_{s\in S(N)}$ to each $S(N)$ . Then the measures 
$\mu_{S(N)}$ converge to $\mu$ in distribution if for any bounded continuous function 
$f:\fh^*_\bR\to\bR$ one has 
\[
\lim_{N\to\infty} \int_{\fh^*_\bR} f(x)d\mu_{S(N)}= \int_{\fh^*_\bR} f(x)d\mu. 
\]
Let $\xi(N)$ be a sequence of random variables. We say that $\xi(N)$ converges in distribution to a random variable $\xi$,
if the measures induced by $\xi(N)$ converge in distribution to the measure induced by $\xi$. We write 
$\xi(N)\to \xi$.

Let $\xi$ be a random variable taking values in $\fh^*_\bR$. The characteristic function $\varphi_{\xi}(t)$, $t\in\fh^*_\bR$ 
is defined as the expectation of 
$e^{i(t,\xi)}$ (for the fixed scalar product). Recall that if $\xi(i)$, $i=1,\dots,k$ are independent 
random variables, then 
$\varphi_{\xi(1)+\dots+\xi(k)}(t)=\prod_{i=1}^k \varphi_{\xi(i)}(t)$.
The importance of the characteristic functions comes from the fact that $\xi(N)\to \xi$ if and only if 
$\lim_{N\to\infty} \varphi_{\xi(N)}(t)=\varphi_{\xi}(t)$ for any $t$.

Let $A$ be an $r\times r$ nondegenerate matrix. Recall the normal distribution $\EuScript{N}(0,A^{-1})$ on the space $\bR^r$ with the density
\[
\frac{\sqrt{\det A}}{(2\pi)^{r/2}} \exp(-\frac{1}{2} (Ax)^*x)
\]
with respect to the standard measure $dx$ on $\bR^r$ (note that we only consider normal distributions with zero mean).
Here and below for two vectors $y=(y_1,\dots,y_r)$ and $x=(x_1,\dots,x_r)$ we write $y^*x=\sum_{i=1}^r y_ix_i$.
Recall that the characteristic function of $\EuScript{N}(0,A^{-1})$ is given by $\exp(-\frac{1}{2} (A^{-1}t)^*t)$.

\section{Weight multiplicities}\label{WM}
Let $\la_1,\dots,\la_k\in P_+$ be a finite set of dominant integral weights and let $\tau_1,\dots,\tau_k$ be a set of rational numbers.
Given a number $N$ such that $N_i=\tau_i N\in\bZ_{\ge 0}$ for all $i$ we consider the tensor product 
\[
\bV_N=\bigotimes_{i=1}^k V_{\la_i}^{\otimes N_i}.
\]  
\begin{rem}
The condition $\tau_i\in\bQ$ is not really important. One can take arbitrary real parameters $\tau_i$ and let the numbers $N_i$
grow according to the law $N_i/N\to\tau_i$.
\end{rem}
The $\fg$-module $\bV_N$ enjoys the weight decomposition $\bV_N=\sum_{\mu\in\fh^*} \bV_N(\mu)$.
The induced discrete measure is defined as
\[
\dim \bV_N^{-1}\sum_{\mu\in\fh^*} \delta_{\mu} \dim \bV_N(\mu) 
\]
(note that $\bV_N(\mu)$ is non trivial only for finite numbers of weights $\mu$). 

Our goal is to study the random variable $\xi(N)$  with the following distribution law:
\begin{equation}\label{Pxi}
P\left(\xi(N)=\frac{\mu}{\sigma N^{1/2}}\right)=\frac{\dim \bV_N(\mu)}{\dim \bV_N},
\end{equation}
where $\sigma$ is certain number depending on $\fg$, $\tau_l$, $\la_l$ to be specified below. 
We first prepare a lemma.

\begin{lem}\label{qt}
Let $(\cdot,\cdot)$ be the standard invariant bilinear nondegenerate form on $\fh^*$ and let $t$ be an
element of $\fh^*$. Then 
\begin{equation}\label{1/N}
\sum_{\mu\in\fh^*} \dim V_\la(\mu)(t,\mu)^2=b_\fg\frac{(\la,\la+2\rho)\dim V_\la}{\dim\fg}(t,t)
\end{equation}
\end{lem}
\begin{proof}
One has the decomposition $t=\sum_{m=1}^r t_m\al_m$.
Let $h_m\in\fh$, $m=1,\dots,r$ be defined by $\mu(h_m)=(\mu,\al_m)$ for any $\mu\in\fh^*$. Then 
\begin{equation} 
\sum_{\mu\in\fh^*} \dim V_\la(\mu)(t,\mu)^2={\rm tr}_{V_\la} (\sum_{m=1}^r t_m h_m)^2.
\end{equation}

For any two elements $x,y\in\fg$ one has 
\[
{\rm tr}_{V_\la} xy=\frac{(\la,\la+2\rho)\dim V_\la}{\dim\fg}b_\fg (x,y).
\] 
In fact, the form ${\rm tr}_{V_\la}(xy)$ defines an invariant nondegenerate form on $\fh$. Therefore,
there exists a constant $u$ such that ${\rm tr}_{V_{\la}} xy=u(x,y)_K=ub_\fg(x,y)$. Now for the Casimir operator $C_2$ defined
via the Killing form one has 
${\rm tr}_{V_{\la}} C_2 = (\la,\la+2\rho)_K\dim V_\la$ (recall that $(\la,\la+2\rho)_K$ is the eigenvalue of $C_2$ on $V_\la$),
and ${\rm tr}_{\fg} C_2 =\dim\fg$, since $C_2$ acts by $1$
in the adjoint representation (see \cite{Kir}, Exercise 8.7).
Hence 
\[
u=\frac{(\la,\la+2\rho)_K\dim V_\la}{\dim\fg}=b_\fg\frac{(\la,\la+2\rho)\dim V_\la}{\dim\fg}
\]
and 
\[
{\rm tr}_{V_\la} h_ih_j=\frac{(\la,\la+2\rho)\dim V_\la}{\dim\fg}(h_i,h_j)=
b_\fg\frac{(\la,\la+2\rho)\dim V_\la}{\dim\fg}(\al_i,\al_j).
\]
We conclude that \eqref{1/N} holds true.
\end{proof}

\begin{rem}
Explicitly, $(t,t)$ equals
$(\bar Ct)^*t$, where $\bar C$ is the (symmetrized) Cartan matrix and $t$ is represented as a vector via
the decomposition into the linear combination of simple roots.
\end{rem}

In what follows we use the notation (see \eqref{Pxi}) 
\begin{equation}\label{sigma}
\sigma^2=b_\fg\frac{\sum_{l=1}^k \tau_l (\la_l,\la_l+2\rho)}{\dim\fg}.
\end{equation} 

\subsection{Simply-laced case}
We first work out the simply-laced case ($\fg$ of type ADE). The general case does not differ much and
will be considered in the subsection below.

\begin{thm}\label{TP}
The sequence of random variables $\xi(N)$ converges in distribution to the random variable $\xi$ with values
$x\in\fh^*_\bR$ with the density
\[
p_\xi(x)=\frac{\sqrt{\det C^{-1}}}{(2\pi)^{r/2}}\exp(-\frac{1}{2}(x,x)),
\]
where $x=\sum_{i=1}^r x_i\om_i$ and the integration is performed against the standard form $dx$ 
Explicitly, $\xi$ has normal distribution $\EuScript{N}(0,C)$.
\end{thm}
\begin{proof}
It suffices to show that $\lim_{N\to\infty} \varphi_{\xi(N)}(t)=\exp(-\frac{1}{2} (Ct)^*t)$ for any $t$.
In fact, let $z=e^{it/\sqrt{N}}$ and for $\mu\in \fh^*$ let 
\[
z^\mu=\exp({\frac{i(t,\mu)}{\sigma\sqrt{N}}}).
\]
Then by definition 
\[
\varphi_{\xi(N)}(t)=(\dim \bV_N)^{-1}\ch \bV_N=\prod_{l=1}^k \left(\frac{\ch V_{\la_l}}{\dim V_{\la_l}}\right)^{N_l}.  
\]
Expanding each exponent $\exp({\frac{i(t,\mu)}{\sigma\sqrt{N}}})$ in $N^{-1/2}$ we obtain
\begin{multline*}
\frac{\ch V_{\la_l}}{\dim V_{\la_l}}= 1 + \frac{iN^{-1/2}}{\sigma\dim V_{\la_l}}(t,\sum_{\mu\in\fh^*} \mu\dim \bV_{\la_l}(\mu)) -\\ 
\frac{N^{-1}}{2\dim V_{\la_l}\sigma^2} \sum_{\mu\in\fh^*} \dim \bV_{\la_l}(\mu)(\mu,t)^2+\cdots.
\end{multline*}
The coefficient of $N^{-1/2}$ vanishes, since for any finite-dimensional $\fg$-module $V$ one has 
\[
\sum_{\mu\in\fh^*} \mu\dim \bV(\mu)=0
\] 
(the left hand side is fixed by $W$).  Lemma \ref{qt} implies that the coefficient in front of $1/N$ is equal to 
\[
\frac{b_\fg(\la_l,\la_l+2\rho)}{\sigma^2 \dim\fg}(t,t).
\]
Using definition \eqref{sigma} and the formula 
\[
\lim_{N\to\infty} (1-\frac{1}{2N} \frac{b_\fg(\la_l,\la_l+2\rho)}{\sigma^2 \dim\fg}(t,t)+\dots)^{\tau_lN}=
\exp(-\frac{1}{2}(t,t)b_\fg\frac{\tau_l(\la_l,\la_l+2\rho)}{\sigma^2 \dim\fg}(t,t))
\]
we conclude that
\[
\lim_{N\to \infty}\varphi_{\xi(N)}(t)=\exp(-\frac{1}{2}(t,t)).
\]  
Finally, we note that in the simply laced case the matrix $C^{-1}$ is the Gram matrix of the scalar products $\om_i,\om_j$.
We conclude that the density of the limiting random variable $\xi$ is equal to 
\[
\frac{\sqrt{\det C^{-1}}}{(2\pi)^{r/2}}\exp(-\frac{1}{2}(x,x)),
\] 
where $x$ is written as a linear combination of fundamental weights $x=\sum_{i=1}^r x_i\om_i$.
\end{proof}

\subsection{The non simply-laced case}
Let $\bar C=DC$ be the symmetrized Cartan matrix, $D={\rm diag}(d_1,\dots,d_r)$. Then the standard invariant form is defined by
$(\al_i,\al_j)=(\bar C)_{i,j}$. In particular, $d_i=(\al_i,\al_i)/2$. The fundamental weights are expressed via the simple
roots by the matrix $(C^t)^{-1}$, i.e. $\om_i=\sum_{j=1}^r (C^t)^{-1}_{i,j}\al_j$. We conclude that the Gram matrix for
the fundamental weights is given by the matrix $(C^t)^{-1}D$. The inverse to this matrix is $D^{-1}C^t$. One easily sees
that this is the Gram matrix of the system of vectors $\al_i d_i^{-1}$. We arrive at the following proposition
(analogue of Theorem \ref{TP}).
\begin{prop}
Theorem \ref{TP} holds true in general if one replaces $C$ with $\bar C$ and assume that the standard invariant form is
defined by the symmetrized Cartan matrix.
\end{prop}
\begin{proof}
One has to write vector $t\in\fh^*_\bR$ from the proof of Theorem \ref{TP} in basis $\al_i d_i^{-1}$. Then the matrix 
responsible for $(t,t)$ is $D^{-1}C^t$ and the inverse matrix is exactly the Gram matrix for the system of fundamental weights.
\end{proof}

\section{Decomposition into irreducibles}\label{DiI}
We now consider another measure on $\fh^*$ coming form the decomposition of the tensor product representation
into irreducible components. Recall $\bV=\bigotimes_{l=1}^k V_{\la_l}^{N_l}$, where $N_l=\tau_lN$. One has
\[
\bV_N=\bigoplus_{\mu\in P_+} V_\mu \otimes M_\mu,
\]
where $M_\mu$ is the space of multiplicities, spanned by the highest weight vectors of all the submodules
$V_\mu$ inside $\bV_N$. In particular, $\dim M_\mu=[\bV_N:V_\mu]$ (the number of times $V_\mu$ show up in $\bV_N$)
and $\dim \bV_N=\sum_{\mu\in P_+} \dim M_\mu\dim V_\mu$. We consider a random variable $\eta(N)$ distributed 
according to the following formula 
\begin{equation}\label{eta}
P\left(\eta(N)=\frac{\mu}{\sigma\sqrt{N}}\right)=\frac{[\bV_N:V_\mu]\dim V_\mu}{\dim \bV_N}.
\end{equation}
We note that $\eta(N)$ takes values only in $\fh^*_+=\sum_{j=1}^r \bR_{\ge 0}\om_j$, i.e. in the 
closure of the fundamental alcove. Our goal in this section
is to study the behaviour of the random variables $\eta(N)$ when $N$ tends to infinity. To this end 
we extend $\eta(N)$ to another random variable $\eta^e(N)$ taking values in the whole $\fh^*_\bR$ by
the following rule. Recall the shifted action of the Weyl group $W$ on $\fh^*$: $w*\beta=w(\beta+\rho)-\rho$.
The group acts transitively on the shifted alcoves and the shifted walls $L_\al-\rho$, where 
$L_\al=\{\beta\in\fh^*_\bR: (\beta,\al)=0\}$ 
labeled by positive roots $\al\in R_+$ form the complement of the union of the alcoves.
We define $\eta^e(N)$ by the following formula: 
\begin{equation}\label{etae}
P\left(\eta^e(N)=\frac{\mu}{\sigma\sqrt{N}}\right)=
\begin{cases}
0, & \mu\in \bigcup_{\al\in R_+} (L_\al-\rho),\\
|W|^{-1}P\left(\eta(N)=\frac{w*\mu}{\sigma\sqrt{N}}\right), & w*\mu\in \fh^*_+. 
\end{cases}
\end{equation}

\begin{example}
Let $\fg=\msl_2$. Then $r=1$, $W=\{e,s\}=S_2$, $L_\al=0$, $P=\bZ$ and $\mu\in P$ can be identified with an integer  number 
(the ration of $\mu$ and $\omega_1$).
Then $s*\mu=-\mu-2$ and $-1$ is fixed by $s$. So definition \eqref{etae} says that 
\[
P\left(\eta^e(N)=\frac{\mu}{\sigma\sqrt{N}}\right)=
\begin{cases}
\frac{1}{2}P\left(\eta(N)=\frac{\mu}{\sigma\sqrt{N}}\right), & \mu\ge 0,\\ 
0, & \mu=-1,\\
\frac{1}{2}P\left(\eta(N)=\frac{-\mu-2}{\sigma\sqrt{N}}\right), & \mu\le -2. 
\end{cases}
\]
\end{example}
We note that the $W$ orbit (with respect to the shifted action) of any element of $P_+$ consists of $W$ elements. 
Hence the definition \eqref{etae} does make sense
(all the probabilities over all $\mu$ sum up to one). 

\begin{rem}
We will compute the density $p_{\eta^e}(x)$, $x\in\fh^*_\bR$ of the limit of the random variables $\eta^e(N)$. 
Then clearly restricting to $\fh^*_+$ one obtains the density 
of the limit of the initial random variables $\eta(N)$. 
\end{rem} 

Let $D$ be an operator on the linear span of the expressions $z^\mu$, $\mu\in\fh^*$ defined by
\begin{equation}
Dz^\mu=z^{\mu-\rho}\frac{\prod_{\al>0} (\mu,\al)}{\prod_{\al>0} (\rho,\al)}.
\end{equation}

\begin{lem}\label{D}
Let $\mu\in P_{+}$, $w\in W$. Then 
$Dz^{w(\mu+\rho)}=(-1)^w z^{w*\mu} \dim V_{\mu}$.
\end{lem}
\begin{proof}
We first note that by the Weyl dimension formula  $Dz^{\mu+\rho}=z^\mu\dim V_\mu$ for $\mu\in P_+$. 

By definition, $Dz^{w(\mu+\rho)}$ is equal to 
\[
z^{w*\mu}\prod_{\al>0} (w(\mu+\rho),\al)=z^{w*\mu}\prod_{\al>0}(\mu+\rho,w\al)=(-1)^wz^{w*\mu}\dim V_\mu,
\]
where the first equality holds since $W$ preserves the scalar product and the second equality holds, since $w$ sends exactly 
$l(w)$ positive roots to negative.
\end{proof}

\begin{prop}\label{DW}
Let $V$ be a finite-dimensional (not necessarily irreducible) representation of $\fg$. Then 
\[
D\left(\ch V \sum_{w\in W} (-1)^w z^{w\rho}\right)=
\sum_{\mu\in P_+}\sum_{w\in W} z^{w*\mu} [V:V_\mu]\dim V_\mu. 
\]  
\end{prop}
\begin{proof}
We start with the case of irreducible $V$. Then $V=V_\la$ for some $\la\in P_+$ and from the Weyl character formula
we obtain that 
\[
\ch V_\la\sum_{w\in W} (-1)^w z^{w\rho}=\sum_{w\in W} (-1)^wz^{w(\la+\rho)}.
\]
Applying $D$ and using Lemma \ref{D} we obtain the desired result. Now for arbitrary $\fg$-module one has the decomposition
$V=\bigoplus_{\la\in P_+} [V:V_\la]V_\la$. Since $D$ is $\bC$-linear, we obtain the desired formula.  
\end{proof}

\begin{thm}\label{TPD}
The random variables $\eta^e(N)$ converge in distribution to a random variable $\eta^e$ with values in $\fh^*$ 
distributed with the density
\[
p_{\eta^e}(x)=|W|^{-1}\frac{\sqrt{\det C^{-1}}}{(2\pi)^{r/2}}\frac{\prod_{\al>0} (x,\al)^2}{\prod_{\al>0}(\rho,\al)}\exp(-\frac{1}{2} (x,x)),
\]
where $x\in\fh^*_\bR$, the integration is performed with respect to the standard form $dx$ and $x$ is identified with 
a vector in $\bR^r$ by writing it as a linear combination of fundamental weights.
\end{thm}
\begin{proof}
Proposition \ref{DW} applied to $V=\bV_N$ implies
\begin{equation}\label{DVN}
\frac{1}{\dim \bV_N}D\left(\ch \bV_N\sum_{w\in W} (-1)^w z^{w\rho}\right)=
\sum_{\substack{\mu\in P_+\\ w\in W}} z^{w*\mu} \frac{[V:V_\mu]\dim V_\mu}{\dim \bV_N}. 
\end{equation}
From definitions \eqref{eta} and \eqref{etae} we see that in order to study the limit of the random variables
$\eta^e(N)$ one needs to understand the behaviour of the left hand side of \eqref{DVN} in the limit $N\to\infty$. 

Let us define numbers $c_\mu$, $\mu\in P$ by the formula
\[
\frac{\ch \bV_N}{\dim \bV_N}=\sum_{\mu\in P} c_\mu z^\mu.
\]
Applying the operator $D$ to $\ch \bV_N (\dim \bV_N)^{-1}\sum_{w\in W} (-1)^w z^{w\rho}$, one gets
\begin{multline*}
\sum_{\substack{\mu\in P\\ w\in W}} (-1)^w c_\mu Dz^{\mu+w\rho}=
\sum_{\substack{\mu\in P\\ w\in W}} (-1)^w c_\mu z^{\mu+w\rho-\rho}\prod_{\al>0} (\mu+w\rho,\al)=\\
\sum_{\nu\in P} z^{\nu-\rho} \prod_{\al>0} (\nu,\al) \sum_{w\in W} (-1)^w c_{\nu-w\rho}.    
\end{multline*}
Let $x=N^{-1/2}\nu/\sigma$. 
Recall that the random variables $\xi(N)$ (see Theorem \ref{TP}) such that
\[
P\left(\xi(N)=\frac{\mu}{\sigma\sqrt{N}}\right)=c_\mu
\] 
converge in distribution to the random variable $\xi$ with the density
\[
p_\xi(x)=\frac{\sqrt{\det C^{-1}}}{(2\pi)^{r/2}}\exp(-(x,x)/2).
\]
Hence if we are able to find the limit
\[
\lim_{N\to\infty} \prod_{\al>0} (xN^{1/2}\sigma,\al) \sum_{w\in W} (-1)^w 
\exp\left(-\frac{(xN^{1/2}\sigma-w\rho,xN^{1/2}\sigma-w\rho)}{2N\sigma^2}\right)
\]
is equal to some $p(x)$, then we will show that the sequence $\eta^e(N)$ converges in distribution
and the density of the limit is given by $\frac{\sqrt{\det C^{-1}}}{|W|(2\pi)^{r/2}}p(x)$.

We rewrite
\begin{equation}\label{ff}
\prod_{\al>0} (\nu,\al)=(N^{1/2}\sigma)^{|R_+|}\prod_{\al>0} (x,\al).
\end{equation}
Now we analyze the factor $\sum_{w\in W} c_{\nu-w\rho}$. Replacing $c_{\nu-w\rho}$ with 
$\exp(-\frac{1}{2N\sigma^2} (\nu-w\rho,\nu-w\rho))$  (see Theorem \ref{TP}) we obtain
\begin{multline*}
\sum_{w\in W} (-1)^w \exp\left(-\frac{(\nu-w\rho,\nu-w\rho)}{2N\sigma^2}\right)=\\
\exp\left(-\frac{(\nu,\nu)}{2N\sigma^2}\right)\exp\left(-\frac{(\rho,\rho)}{2N\sigma^2}\right)
\sum_{w\in W} (-1)^w \exp\left(-\frac{(2\nu,w\rho)}{2N\sigma^2}\right).
\end{multline*}
The factor $\exp\left(-\frac{(\nu,\nu)}{2N\sigma^2}\right)$ after the substitution $x=N^{-1/2}\nu/\sigma$ turns into
$\exp(-\frac{1}{2}(x,x))$. The factor $\exp\left(-\frac{(\rho,\rho)}{2N\sigma^2}\right)$ tends to $1$ when $N\to\infty$.
The most complicated is the last factor 
\[
\sum_{w\in W} (-1)^w \exp\left(-\frac{(2\nu,w\rho)}{2N\sigma^2}\right).
\]

For a weight $\nu$ we define a linear functional $\pi_\nu$ on the $\bC[P]$ (the group algebra of the weight lattice $P$)
by the formula 
$\pi_\nu(z^{\beta})=\exp\left(-\frac{(2\nu,\beta)}{2N\sigma^2}\right)$. Then
\begin{multline*}
\sum_{w\in W} (-1)^w \exp\left(-\frac{(2\nu,w\rho)}{2N\sigma^2}\right)=\pi_\nu \sum_{w\in W} (-1)^w z^{w\rho}=\\
\pi_\nu \prod_{\al>0} (z^{\al/2}-z^{\al/2})=
\prod_{\al>0} \left(\exp\left(\frac{(\nu,\al)}{2N\sigma^2}\right)-\exp\left(-\frac{(\nu,\al)}{2N\sigma^2}\right)\right).
\end{multline*}
Expanding in $1/N$ we obtain the leading term
\[
(\sigma N^{1/2})^{-|R_+|}\prod_{\al>0} \left(\frac{\nu}{N^{1/2}\sigma},\al\right)=(\sigma N^{1/2})^{-|R_+|}\prod_{\al>0} (x,\al). 
\]
Multiplying by \eqref{ff}, we arrive at the desired formula.
\end{proof}

\begin{cor}
The random variables $\eta(N)$ converge in distribution to the random variable $\eta$ taking values in $\fh^*_+$
with the density given by
\[
p_\eta(x)=\frac{\sqrt{\det C^{-1}}}{(2\pi)^{r/2}}\frac{\prod_{\al>0} (x,\al)^2}{\prod_{\al>0}(\rho,\al)}\exp(-\frac{1}{2} (x,x)).
\]
\end{cor}
\begin{proof}
It suffices to note that $p_{\eta^e}(x)$ is $W$ invariant.
\end{proof}
\begin{rem}
In the non simply-laced case one has to replace $C$ with the symmetrized Cartan  matrix $\bar C$. 
\end{rem}

\section{Type A}\label{TypeA}
In this section we consider the Lie algebra 
$\msl_n$ of rank $r=n-1$ (of type $A_{n-1}$). Let $\al_i$, $\om_i$, $i=1,\dots,n-1$ be simple roots and fundamental weights.
The positive roots of $\msl_n$ are of the form $\al_{i,j}=\al_i+\al_{i+1}+\dots+\al_{j-1}$, $1\le i<j\le n$.  
The nonzero entries of the Cartan matrix $C$ are given by $c_{i,i}=2$, $c_{i,i+1}=c_{i+1,i}=-1$. The inverse matrix  
is the symmetric matrix defined by the formula $(C^{-1})_{i,j}=\frac{i(n-j)}{n}$ for $1\le i\le j\le n-1$.

The limiting density for the random variables $\xi(N)$ (see Theorem \ref{TP})  is  given by the following function on 
$\bR^{n-1}$:
\[
\frac{\sqrt{1/n}}{(2\pi)^{\frac{n-1}{2}}}\exp\left(-\frac{1}{2}\sum_{1\le i,j\le n-1} (C^{-1})_{i,j}x_ix_j\right)
\]
and the limiting density for the random variables $\eta(N)$ (see Theorem \ref{TPD})  is  given by the following function on 
$\bR_{\ge 0}^{n-1}$:
\[
\frac{\sqrt{1/n}}{(2\pi)^\frac{n-1}{2}}\frac{\prod_{1\le i\le j\le n-1} (x_i+\dots+x_j)^2}{\prod_{1\le i\le j\le n-1} (j-i+1)}
\exp\left(-\frac{1}{2}\sum_{1\le i,j\le n-1} (C^{-1})_{i,j}x_ix_j\right).
\]

In the simplest case of $\msl_2$ the densities reduce to
\[
\frac{\sqrt{1/2}}{\sqrt{2\pi}}\exp\left(-\frac{1}{4}x^2\right), x\in\bR \text{ and }
\frac{\sqrt{1/2}}{\sqrt{2\pi}}x^2\exp\left(-\frac{1}{4}x^2\right), x\ge 0.
\]

Recall that the density for the joint GUE eigenvalues distribution equals (up to a scalar factor) 
$\exp(-\frac{1}{2}\sum_{k=1}^n a_k^2)\prod_{1\le i<j\le n} (a_i-a_j)^2$, where $a_1\ge\dots \ge a_n$
are eigenvalues. In \cite{K1} Kerov showed that this density pops up in the limit of the distributions
coming from the decomposition of the large tensor powers of vector representation of $\mgl_n$ into irreducible
components. Let us rewrite this density in Lie theoretic terms. 

\begin{prop}
Assume that $\sum_{k=1}^n a_k=0$. Let $x_j=a_j-a_{j+1}$, $j=1,\dots,n-1$. Then 
\[
\exp\left(-\frac{1}{2}\sum_{k=1}^n a_k^2\right)\prod_{1\le i<j\le n} (a_i-a_j)^2 = \exp(-(x,x)/2)\prod_{\al>0} (x,\al)^2,
\]
where $x=\sum_{j=1}^{n-1} x_j\omega_j$.
\end{prop}
\begin{proof}
Since the positive roots of $\msl_n$ are of the form $\al_{i,j}=\al_i+\al_{i+1}+\dots+\al_{j-1}$, $1\le i<j\le n$, one
has $(x,\al_{i,j})=x_i+x_{i+1}+\dots+x_{j-1}=a_i-a_j$.  Now it's suffices to show that
 \[
\sum_{k=1}^n a_k^2 - \frac{1}{n} (a_1+\dots +a_n)^2 = \sum_{1\le i,j\le n-1} (C^{-1})_{i,j}x_ix_j,
\]
which is the direct computation.
\end{proof}

\end{document}